\documentclass[12pt]{amsart}
\usepackage[active]{srcltx}
\usepackage{a4wide}
\usepackage{amsthm,amsfonts,amsmath,mathrsfs,amssymb}
\usepackage{dsfont}

\newtheorem*{theorem*}{Theorem}
\newtheorem{theorem}{Theorem}

\newtheorem*{proposition*}{Proposition}
\newtheorem{lemma}{Lemma}

\theoremstyle{remark}

\newcommand{\nc}{\newcommand}

\newcommand{\C}{{\mathbb C}}

\newcommand{\D}{{\mathbb D}}

\newcommand{\T}{{\mathbb T}}
\newcommand{\ve}{{\varepsilon}}

\nc{\supp}{\operatorname{supp}}
\nc{\dif}{\operatorname{d}} \nc{\im}{\operatorname{i}}
\nc{\Hi}{{\mathscr{H}}^\infty} \nc{\Ht}{{\mathscr{H}}^2}
\nc{\Hone}{{\mathscr{H}}^1} \nc{\ol}{\overline} \nc{\bz}{\mathbf{z}}
\nc{\bw}{\mathbf{w}} \nc{\eps}{\varepsilon}

\DeclareMathOperator {\var}{Var_{[0,1]}}
\DeclareMathOperator {\lip12}{Lip_{1/2}}
\DeclareMathOperator {\lipa}{Lip_{\alpha}}
\DeclareMathOperator {\bv}{BV}

\allowdisplaybreaks[3]

\title[GCD sums from Poisson integrals and systems of dilated functions]{GCD sums from Poisson integrals
\\ and  systems of dilated functions}

\author{Christoph Aistleitner}
\address{Graz University of Technology,
Institute of Mathematics A, Steyrergasse 30, 8010 Graz, Austria}
\email{aistleitner@math.tugraz.at}

\author{Istv\'an Berkes}
\address{Graz University of Technology,
Institute of Statistics, Keplergasse 24, 8010 Graz, Austria}
\email{berkes@tugraz.at}

\author{Kristian Seip}
\address{Department of Mathematical Sciences, Norwegian University of
Science and Technology (NTNU), NO-7491 Trondheim, Norway}
\email{seip@math.ntnu.no}

\thanks{Research supported by a Schr\"odinger scholarship of the Austrian Research Foundation (FWF) (Aistleitner), FWF Grant P24302-N18
and OTKA Grant K 108615 (Berkes) and  Grant 185359/V30 of the Research Council of Norway (Seip). This paper was written while the second and third
author participated in the research program \emph{Operator Related Function Theory and Time-Frequency Analysis} at the Centre for Advanced
Study at the Norwegian Academy of Science and Letters in Oslo during 2012--2013.}
\date{\today}


\begin{document}

\begin{abstract}
Upper bounds for GCD sums of the form
$$
\sum_{k,{\ell}=1}^N\frac{(\gcd(n_k,n_{\ell}))^{2\alpha}}{(n_k
n_{\ell})^\alpha}
$$
are established, where $(n_k)_{1 \leq k \leq N}$ is any sequence of distinct positive integers and $0<\alpha \le 1$; the estimate for $\alpha=1/2$ solves in particular a problem of Dyer and Harman from 1986, and the estimates are optimal except possibly for $\alpha=1/2$. The method of proof is based on identifying the sum as a certain Poisson integral on a polydisc; as a byproduct, estimates for the largest eigenvalues of the associated GCD matrices are also found.  The bounds for such GCD sums are used to establish a Carleson--Hunt-type inequality for systems of dilated functions of bounded variation or belonging to $\lip12$, a result that in turn settles two longstanding problems on the a.e.\ behavior of systems of dilated functions: the a.e. growth of sums of the form $\sum_{k=1}^N f(n_k x)$ and the a.e.\ convergence of $\sum_{k=1}^\infty c_k f(n_kx)$ when $f$ is $1$-periodic and of bounded variation or in  $\lip12$.
\end{abstract}

\maketitle

\section{Introduction} \label{sect1}
This paper studies two closely related topics: Greatest common divisor (GCD) sums of the form
\begin{equation}\label{gcda}
\sum_{k,\ell=1}^N\frac{(\gcd(n_k,n_{\ell}))^{2\alpha}}{(n_k
n_{\ell})^\alpha}
\end{equation}
for $0<\alpha \le 1$ and convergence properties of systems of dilated functions $f(n_k x)$ on the unit interval $[0,1]$. Here $(n_k)_{k \geq 1}$ is a sequence of distinct positive integers and $f$ is a $1$-periodic real-valued function of bounded variation or belonging to the class $\lip12$. We will introduce a new method for estimating sums of the form \eqref{gcda} and in particular solve a problem posed by Dyer and Harman in \cite{DH}. In addition, using estimates for \eqref{gcda}, we will establish a version of the Carleson--Hunt inequality that settles two longstanding problems regarding the a.e.\ behavior of systems of dilated functions.

The study of GCD sums like \eqref{gcda} was initiated by Koksma who in the 1930s observed that such sums
can be used to estimate integrals of the form
\begin{equation}\label{kok_int}
\int_0^1 \left( \sum_{k=1}^N \left(\mathds{1}_{[a,b)} (\{n_k x \})-(b-a)\right)  \right)^2 dx,
\end{equation}
where the notation $\{ \cdot \}$ stands for fractional part. Integrals like \eqref{kok_int} give in turn important information about the distribution of the sequence $(\{n_k x\})_{k \geq 1}$ for almost all $x \in (0, 1)$. In the case $\alpha=1$, G\'{a}l \cite{G} proved that\footnote{Here and in what follows we may assume that $N\ge 3$ so that $\log\log N$ is well defined and positive.}
\begin{equation} \label{galsums}
\frac{1}{N}\sum_{k,\ell=1}^N\frac{(\gcd(n_k,n_{\ell}))^{2}}{n_k
n_{\ell}} \leq c (\log \log N)^2,
\end{equation}
and he showed that this bound is optimal up to the value of the absolute constant $c$.
In 1986, Dyer and Harman \cite{DH} proved that
\begin{equation} \label{dhsums}
\frac{1}{N}\sum_{k,\ell=1}^N\frac{\gcd(n_k,n_{\ell})}{\sqrt{n_k
n_{\ell}}} \leq C \exp \left( \frac{c \log N}{\log \log N} \right)
\end{equation}
for two absolute constants $C$ and $c$, and they used this estimate to prove results
in metric Diophantine approximation; Dyer and Harman found also that
$$
\frac{1}{N} \sum_{k,\ell=1}^N\frac{(\gcd(n_k,n_{\ell}))^{2\alpha}}{(n_k
n_{\ell})^\alpha} \leq c(\alpha) \exp \left((\log N)^{(4-4\alpha)/(3-2 \alpha)} \right)
$$
for $1/2<\alpha<1$. In his monograph \cite{harman}, Harman writes that ``it is tempting to conjecture'' that the right-hand side of \eqref{dhsums} can be replaced by a constant times $\exp\! \big(c\sqrt{\log N}/\! \log \log N\! \big)$. One of our examples given below will disprove this conjecture and show that here we can not have a function smaller than $\exp\! \big(2\sqrt{(\log N)/\!\log \log N} \big)$.  However, the following theorem, which is our main result on GCD sums, will ``almost'' confirm Harman's conjecture and yield optimal upper bounds for \eqref{gcda} when $1/2<\alpha <1$.
\begin{theorem} \label{gcd}For every $\varepsilon>0$, there exists a positive constant $C_{\varepsilon}$ such that the following holds. For $0<\alpha<1$ and an arbitrary $N$-tuple of distinct positive integers $n_1,
n_2,..., n_N$, we have
\[  \frac{1}{N}\sum_{k,\ell=1}^N\frac{(\gcd(n_k,n_{\ell}))^{2\alpha}}{(n_k
n_{\ell})^\alpha} \le  
                                              C_{\varepsilon}\exp\left((1+\varepsilon)g(\alpha,N)
                                             \right),  \]
where 
\[ g(\alpha,N)=\begin{cases}  
                                              \left(\frac{8}{1-\alpha}+\frac{16\cdot 2^{-\alpha}}{\sqrt{2\alpha-1}}\right) (\log N)^{1-\alpha}(\log\log N)^{-\alpha}
                                               + \frac{1}{1-\alpha}(\log N)^{(1-\alpha)/2}, \ \  1/2 < \alpha < 1 \\
50\alpha (\log N \log\log N)^{1/2}+(1-2\alpha)\log N, \ \ \ \ \ \quad  0<\alpha\le 1/2.
  \end{cases} \]\end{theorem}


Theorem~\ref{gcd} is in fact a corollary to a more general result which can be given a function theoretic interpretation on the infinite-dimensional polydisc $\D^{\infty}$. The observation underlying this general theorem is that the GCD sum \eqref{gcda} can be written as a certain Poisson integral evaluated at the point $(p_j^{-\alpha})$ in $\D^{\infty}$, where $p_j$ denotes the $j$-th prime number. Such integrals can be computed for arbitrary points in $\D^{\infty}$, and our theorem is roughly speaking stated in this generality. The proof requires a surprising blend of an intricate combinatorial argument found in G\'{a}l's work \cite{G} and the explicit expression for the Poisson kernel on polydiscs. Thus number theory plays a minor role in establishing Theorem~\ref{gcd} and enters the discussion only at the final point, where we need information about the decay of the sequence $(p_j^{-\alpha})$.

We will show by an example that Theorem~\ref{gcd} is best possible (up to a constant factor in the exponent) when $1/2<\alpha<1$. We will also see that the blow-up of the constant in front of the leading term in $g(\alpha,N)$  is of the right magnitude when $\alpha\nearrow 1$. We conjecture that the blow-up of the same constant when $\alpha\searrow 1/2$ is an artifact and that the estimate in the range $1/2<\alpha <1$ should indeed extend to $\alpha=1/2$, which would then be optimal too. On the other hand, as we will see, the estimates change abruptly when we pass from $\alpha=1/2$ to $\alpha<1/2$, as a consequence of the divergence of the series $\sum p_j^{-2\alpha}$; the slow divergence when $\alpha=1/2$ is the reason why this is a particularly delicate case. The range $0<\alpha<1/2$, included here for the sake of completeness, is less subtle, and it is easy to give an example showing that the estimate of Theorem~\ref{gcd} is essentially best possible.



The proof of Theorem~\ref{gcd} and the examples showing that our
results are essentially optimal  will be presented in
Section~\ref{sect3} below. An immediate consequence of our
reformulation in terms of Poisson integrals is that the
corresponding matrices are positive definite. In the subsequent
Section~\ref{Sect4}, we will see that in turn Theorem~\ref{gcd}
implies precise estimates for the largest eigenvalues of these
matrices, or, equivalently, for their spectral norms.

\section{Applications to systems of dilated functions} \label{sect2}

Our main application of Theorem~\ref{gcd}, to be found in Section~5 below, will be to establish a Carleson--Hunt-type inequality for systems of dilated functions of bounded variation or belonging to $\lip12$. By standard arguments, this inequality will yield asymptotically precise results for
the growth of 
\begin{equation}\label{fnkx}
\sum_{k=1}^N f(n_k x)
\end{equation}
and for the almost everywhere convergence of 
\begin{equation}\label{fseries}
\sum_{k=1}^\infty c_k f(n_kx)
\end{equation}
for functions $f$ of bounded variation or belonging to $\lip12$ that satisfy
\begin{equation} \label{f1}
f(x+1)=f(x), \qquad  \int_0^1 f(x)\, dx = 0.
\end{equation}
Such dilated sums arise in many problems in analytic number theory, Diophantine approximation,  uniform distribution
theory, harmonic analysis, ergodic theory, and probability theory. Estimating the sum (\ref{fnkx}) for centered indicator
functions $f=f_{a, b}= \chi_{(a, b)}-(b-a)$, which are extended with period 1, is equivalent to measuring the uniformity (more precisely the deviation
from uniformity) of the distribution of the sequence $(n_k x)_{k \geq 1}$ modulo $1$, and for $n_k=k$ very precise results are known. Khinchin
\cite{kh1} proved that the discrepancy of the sequence $(kx)_{1\le k \le N}$ satisfies
$$ N D_N(x, 2x, ..., NX) \ll (\log N)^{1+\varepsilon}\qquad \text{a.e.} $$
for every $\varepsilon>0$ and that this becomes false for $\varepsilon=0$. Here the discrepancy $D_N(x_1, \dots,x_N)$ of a sequence $x_1, \dots, x_N$ of real numbers is defined as
$$
D_N(x_1, \dots, x_N) = \sup_{0 \leq a \leq b \leq 1} \left| \frac{1}{N} \sum_{k=1}^N f_{a,b} (x_k) \right|,
$$
where again $f_{a,b}$ denotes the centered indicator function of the interval $(a,b) \subset [0,1]$, extended with period 1. Thus we have
\begin{equation} \label{kx}
\left|\sum_{k=1}^N f_{a,b}(kx)\right| \ll  (\log N)^{1+\varepsilon}\qquad \text{a.e.}
\end{equation}
uniformly for such centered indicators $f_{a, b}$, and, in view of Koksma's inequality  (see e.g.\ \cite{KN}, p.\ 143), uniformly for all 1-periodic functions $f$ satisfying (\ref{f1}) and
$\text{Var}_{[0, 1]} (f)\le 1$.  In view of Schmidt's lower bound \cite{sch} for the discrepancy
of arbitrary infinite sequences, the metric discrepancy behavior
of $(k x)_{k\ge 1}$ is near to extremal.

For general $(n_k)_{k\ge 1}$, the situation changes markedly.
For $f(x)=2\chi_{[0, 1/2)}(x)-1$ (extended to $\mathbb R$ with period 1) and $n_k=2^k$, the terms of
(5) reduce to the Rademacher functions, and the law of the iterated logarithm implies that for almost all
$x\in (0, 1)$ the sum (\ref{fnkx}) exceeds $(N\log\log N)^{1/2}$ for infinitely many $N$.
Berkes and Philipp \cite{beph}
constructed a sequence $(n_k)_{k\ge 1}$ such that for $f(x)=\{x\}-1/2$ and for almost all $x$ the relation
\begin{equation}\label{bp1994}
\left|\sum_{k=1}^N f(n_kx)\right| \ge (N \log N)^{1/2}
\end{equation}
holds for infinitely many $N$, providing an even faster growing sum (\ref{fnkx}).
In the opposite direction, R.C.\ Baker \cite{baker} showed, improving
earlier results of Cassels \cite{cas} and Erd\H{o}s and Koksma \cite{ek},
that for every increasing sequence $(n_k)_{k\ge 1}$ of integers, the
discrepancy of the sequence $(n_kx)_{1\le k \le
N}$ satisfies
\begin{equation} \label {baker}
D_N (n_1 x, ..., n_N x) \ll N^{-1/2} (\log N)^{3/2+\ve} \quad \text{a.e.}
\end{equation}
for every $\ve>0$. As a consequence, we have
\begin{equation}\label{32}
\left|\sum_{k=1}^N f(n_kx)\right| \ll \sqrt{N} (\log N)^{3/2+\ve}
\quad \text{a.e.}
\end{equation}
uniformly for all $f$ satisfying (\ref{f1}) and $\text{Var}_{[0, 1]} (f)\le 1$.
There is a gap between (\ref{bp1994}) and (\ref{32}); in particular it is not known if the uniform estimate
\eqref{32} holds for $\ve=0$ and all $(n_k)_{k\ge 1}$.
For a fixed $f\in \bv$ (i.e.\ without uniformity), Aistleitner, Mayer, and Ziegler \cite{AMZ}
improved the upper bound in (\ref{32}) to
$$\mathcal{O} \Big(\sqrt{N} (\log N)^{3/2} (\log\log N)^{-1/2+\ve}\Big),$$
getting for the first time a bound better than $\mathcal{O}(\sqrt{N} (\log N)^{3/2})$. (Here, and in the sequel, we write
$f \in \bv$ if $\var f < \infty.$) Our Carleson--Hunt-type inequality will give the following improvement of this estimate.

\begin{theorem} \label{a1}
Let $(n_k)_{k \geq 1}$ be a strictly increasing sequence of positive
integers, let  $f$ be a function satisfying \eqref{f1}, and assume in
addition that either $f \in \bv$ or $f \in \lip12$. Then for every
$\ve>0$,
\begin{equation}\label{1/2est}
\left| \sum_{k=1}^N f(n_k x) \right| \ll (N\log
N)^{1/2} (\log\log N)^{5/2+\ve} \qquad \textup{a.e.}
\end{equation}
when $N \to \infty$.
\end{theorem}

This estimate is sharp up to the exact value of the exponent of
$\log\log N$, as shown by the
following result of Berkes and Philipp \cite[Theorem 1]{beph}:
There exists an increasing sequence $(n_k)_{k\ge 1}$ such that
\[ \limsup_{N\to\infty} \frac{\left| \sum_{k=1}^N \cos (2\pi n_k x) \right|}
{(N\log N \log\log N)^{1/2}}=\infty \qquad
\textup{a.e.} \]

The class $\lip12$ represents an interesting limiting case in this context.
Kaufman and Philipp \cite{kp} proved that, under the lacunarity condition  $n_{k+1}/n_k\ge q>1$ $(k=1, 2, \ldots)$,
the law of the iterated logarithm
\begin{equation}\label{phlil}
\left| \sum_{k=1}^N f(n_kx)\right| \ll (N\log\log N)^{1/2} \qquad \text{a.e.}
\end{equation}
holds uniformly for all $f\in \lipa$, $\alpha>1/2$, with a fixed Lipschitz constant, and this fails for
$\alpha<1/2$. The case $\alpha=1/2$ remains open.
In the case of Theorem \ref{a1}, the proof shows that for $f\in \lipa$, $\alpha> 1/2$, the exponent
5/2 in (\ref{1/2est}) can be  replaced by 1/2 and
this exponent is best possible.

The second consequence of our version of the Carleson--Hunt inequality deals with the a.e.\ convergence of series of the form
\begin{equation} \label{fconv}
\sum_{k=1}^\infty c_k f(n_kx)
\end{equation}
for 1-periodic functions $f$.
By Carleson's theorem \cite{carl}, when $f(x)=\sin 2\pi x$ or $f(x)=\cos 2\pi x$, the series
(\ref{fconv}) converges a.e.\ provided that $\sum_{k=1}^\infty
c_k^2<\infty$. Gaposhkin \cite{gapo} showed that this remains valid
if the Fourier series of $f$ converges absolutely; in particular,
this holds if $f$ belongs to the class $\lipa$ for some
$\alpha>1/2$. However, Nikishin \cite{ni} showed that the analogue
of Carleson's theorem fails for $f(x)=\hbox{sgn} \sin 2\pi x$, and
it also fails for some continuous function $f$.
There is an extensive literature on this convergence problem going back to
the 1940s (see \cite{bewe} and \cite{ga66b} for the history of the subject),
and sufficient a.e.\ convergence criteria
have been obtained for various classes of functions such
as $\lipa$, $0<\alpha\le 1/2$, $L^p$, $BV$, or spaces of functions defined
via decay conditions on Fourier coefficients, see e.g.
\cite{alipa, bewe, bewe2, bewe3, bremont, ga66b, gapo2, weber}. However, except for Carleson's
theorem and its immediate consequences, no precise a.e.\ convergence criteria for the series
(\ref{fconv}) have been found. The following theorem gives an essentially complete
solution to the convergence problem for $\bv$ and a substantial improvement of known results for the
class $\lip12$.

\begin{theorem} \label{a2}
Let $f$ be a function satisfying \eqref{f1} and assume in addition
that either  $f \in \bv$ or $f \in \lip12$.  Let $(c_k)_{k \geq 1}$ be a real
sequence satisfying
\begin{equation}\label{2}
\sum_{k=3}^\infty  c_k^2 (\log\log k)^\gamma<\infty
\end{equation}
for some $\gamma> 4$. Then for every increasing sequence $(n_k)_{k \geq 1}$ of
positive  integers  the series $\sum_{k=1}^\infty c_k f(n_kx)$
converges a.e.
\end{theorem}

Using the optimality of G\'{a}l's theorem and a probabilistic argument, we will in Section~\ref{section6}
show that for every $0<\gamma< 2$ there exists an increasing sequence $(n_k)_{k\ge 1}$ of positive integers and a real
sequence $(c_k)_{k\ge 1}$  such that \eqref{2} holds, but $\sum_{k=1}^\infty c_k f(n_kx)$ is a.e.\ divergent for
$f(x)=\{x\}-1/2$. Thus apart from the precise value of the exponent of $\log \log k$, Theorem~\ref{a2} is best possible for $f\in\bv$. In the $\lip12$ case, the argument in Section~\ref{section6} gives a slightly weaker counterexample, with
$\log\log k$ in (\ref{2}) replaced by $\log\log\log k$. On the other hand,
in the case of $f \in \lipa$, $0<\alpha<1/2$, Theorem~3 of \cite{berkes} gives an a.e.\ divergent
series (\ref{fseries})  with
$$\sum_{k=1}^\infty c_k^2 (\log k)^\gamma<\infty \qquad \text{for all} \ 0<\gamma<1-2\alpha.$$
Comparing this result with Theorem~\ref{a2}, we see that there is an essential difference between the convergence
behavior of the sum \eqref{fnkx} for $\alpha=1/2$ and $\alpha<1/2$. We conclude again that $\lip12$ stands out as a particularly interesting limiting case.

We mention finally two additional applications of Theorem \ref{gcd}. First, we may obtain a substantial improvement of the convergence criteria
in \cite{alipa} and \cite{weber} for the case $0<\alpha<1/2$; we will discuss this problem in a
subsequent paper. Second, Theorem \ref{gcd} yields an improvement of
a result of  Harman \cite{harmand} on metric Diophantine
approximation.
The effect of replacing the estimate \eqref{dhsums} in Harman's original proof by our Theorem \ref{gcd} is that a
factor of order $\exp \left(c \log N/\log \log N\right)$ becomes instead a factor
of order $\exp\left(c \sqrt{\log N\log \log N} \right)$. This result
is connected with the Duffin--Schaeffer conjecture, a notoriously
difficult open problem from metric Diophantine approximation (see
\cite{harman,harmans}).

\section{Proof of Theorem~\ref{gcd} via  trigonometric polynomials on $\D^\infty$}\label{sect3}

We introduce multi-index notation suitable for our purposes. A multi-index is a sequence $\beta=
(\beta^{(1)}, \beta^{(2)}, ..., \beta^{(R)}, 0, 0, ...)$ consisting of nonnegative integers with only a finite number of them being nonzero. We let $\supp \beta$ be the finite set of positive integers $j$ for which $\beta^{(j)}>0$; we write $R(\beta)$ for the maximal element in $\supp \beta$. Two multi-indices $\beta$ and $\mu$ may be added and subtracted as sequences. Then $\beta-\mu$ may fail to be a multi-index, but the sequence $|\beta-\mu|=(|\beta^{(j)}-\mu^{(j)}|)$ will again be a multi-index. We may multiply multi-indices by positive integers in the obvious way and express any multi-index as a linear combination of the natural basis elements $e_j$, where $e_j$ is the multi-index supported by $\{j\}$ with $e^{(j)}_j=1$.  We write
$\beta\le \mu$ if $\beta^{(j)}\le \mu^{(j)}$ for every $j$. For a sequence of complex numbers $z=(z_j)$, we  use the notation
\[ z^\beta:=\prod_{j=1}^{R(\beta)} z_j^{\beta^{(j)}};\]
we will sometimes write $z^{-\beta}$ for the number $(z^{\beta})^{-1}$.

We write $p=(p_j)$ for the sequence of prime numbers ordered by ascending magnitude. Using our multi-index notation, we may write every positive integer $n$ as $p^{\beta}$ for a multi-index $\beta$ that is uniquely determined by $n$. If $n_k=p^{\beta_k}$, then we may write
\[   \frac{(\operatorname{gcd}(n_k,n_{\ell}))^{2}}{n_k
n_{\ell}} = p^{-|\beta_k-\beta_{\ell} |}.\]
For an arbitrary sequence $t$ of positive numbers in $\D^{\infty}$ and a set of distinct multi-indices $B=\{\beta_1,..., \beta_N\}$, we now define
\[ S(t,B):=\frac{1}{N}\sum_{k,{\ell}=1}^N t^{|\beta_k-\beta_{\ell} |}. \]
We set
\[ \Gamma_t(N):=\sup_{B} S(t,B),\]
where the supremum is taken over all possible sets $B$ of distinct multi-indices $\beta_1,...,\beta_N$. Our original problem concerning GCD sums has thus been transformed into the problem of estimating $\Gamma_t(N)$ in the particular case when $t=(p_j^{-\alpha})$.


For a minor technical reason, we introduce the following notation. Let $\eta:(0,1)\to (0,1)$ be defined by the relation
\[  \eta(x):=\begin{cases} 2x, & 0<x <1/2 \\
x, & 1/2 \le x <1, \end{cases} \]
and for a sequence $t=(t_j)$ with $0<t_j<1$, we set $\eta(t):=(\eta(t_j))$. For a decreasing sequence $t$ of positive numbers in the sequence space $c_0$, we define
\[ \kappa(t):=\begin{cases} 0 & \text{if} \ t_1<1/2 \\
\max\{j:\ t_j\ge 1/2\} & \text{otherwise}. \end{cases} \]

We will prove the following general theorem.

\begin{theorem}\label{general}
Let $t=(t_j)$ be a sequence of positive numbers in $\D^{\infty}\cap c_0$ such that $(\tau_j):=\eta(t)$ is a decreasing sequence. Fix a positive number $\xi>(\log 2)^{-1}$, and set $r_N=[\xi \log N]+\kappa(t)$. Then, for  arbitrary numbers $1>v_1\ge v_2 \ge \cdots \ge v_{r_N}$ satisfying also $v_j > \tau_j^2 $ for $1\le j  \le r_N$, we have
\begin{equation}\label{th4form}
\Gamma_{t}(N) \le \prod_{j=1}^{r_N}(1-v_j)^{-1}(1-v_j^{-1} \tau_j^2)^{-1} \prod_{k=r_N+1}^{N-1}(1-v_{r(N)}^{-1} \tau_k^2)^{-1}+\exp\left(C
\sum_{\ell=1}^{N-1}t_{\ell}^2\right),
\end{equation}
where $C$ is a positive constant depending only on $\xi$.
\end{theorem}
This theorem is clearly applicable when the sequence $t$ is in $\ell^2$, but it can also be used when the series $\sum_j t_j^2$ is ``slowly'' divergent, as we will now see.

\begin{proof}[Proof of Theorem~\ref{gcd}]
We now take Theorem~\ref{general} for granted and show that it implies Theorem~\ref{gcd}. We begin with the case $1/2<\alpha < 1$ and observe first that then
\[  \exp\left(C \sum_{\ell=1}^{N-1}t_{\ell}^2\right)\le \left(\exp(c+C \min( \log\log N, 1/(2\alpha-1))\right) \]
for some constant $c$. This inequality has the consequence that the exponential term in \eqref{th4form} will contribute only with a fixed constant factor, independent of $\varepsilon$, in $C_{\varepsilon}$. Assuming that $N$ is so large that $(2\alpha-1)^{-1/2}\le \tau_{1}$, we  choose
\[v_j:=\max(\tau_j, (2\alpha-1)^{-1/2} \tau_{r_N})\] in the first term on the right-hand side of \eqref{th4form}, with $\tau_j=\eta( p_j^{-\alpha})$. (The decay of $\tau_j$ is a minor technical point which can be dealt with by an obvious rearrangement of the sequence. For smaller $N$, we set $v_j:= \tau_0$ for all $j$. We choose $\xi=2$ and note that $p_j^{-\alpha}< 1/2$ for $j\ge 3$, whence we have $\tau_j=2p_j^{-\alpha}$ for $j\ge 3$ and $r_N=[2\log N] +2$. We set
$$s_N:=\max \left\{1\le j\le r_N:\ \tau_j\ge (2\alpha-1)^{-1/2} \tau_{r_N}\right\}$$
and split accordingly the first product into two factors. Hence, using the definition of $\tau_j$, we obtain
\begin{equation}\label{combas}
\Pi_1:=\prod_{j=j_0}^{s_N} (1-v_j)^{-1}(1-v_j^{-1} \tau_j^2)^{-1} \le  \prod_{j=j_0}^{r_N}(1-2p_j^{-\alpha})^{-2}
\le \exp\left( (1+\varepsilon/2)4 \sum_{j=j_0}^{r_N} p_j^{-\alpha}\right)
\end{equation}
and
\begin{eqnarray} \nonumber \Pi_2:=\prod_{j=s_N+1}^{r_N}(1-v_j)^{-1}(1-v_j^{-1} \tau_j^2)^{-1} & \le & (1-\min(\tau_0,(2\alpha-1)^{-1/2}2p_{r_N}^{-\alpha}))^{-2r_N} \\ \label{combas2}
& \le & C \exp\left((1+\varepsilon/2)8(2\alpha-1)^{-1/2}p_{r_N}^{-\alpha} \log N\right)
\end{eqnarray}
if $j_0$ and thus $s_N$ are large enough, with $C$ an absolute constant.
By the prime number theorem, we have $p_j=(1+o(1))j \log j$ when $j\to \infty$, so that \eqref{combas} and \eqref{combas2} become respectively
\begin{equation}\label{esti1}
\Pi_1 \le  \exp\left( (1+\varepsilon)4 \sum_{j=j_0}^{r_N} (j \log j)^{-\alpha} \right)
\end{equation}
and
\begin{equation}\label{esti2}
\Pi_2 \le  C \exp\left((1+\varepsilon)8\cdot 2^{-\alpha}(2\alpha-1)^{-1/2}(\log N)^{1-\alpha}(\log\log N)^{-\alpha}\right)
\end{equation}
if $j_0$ is large enough. The sum in \eqref{esti1} can be estimated as
\[ \sum_{j=j_0}^{r_N} (j \log j)^{-\alpha}\le (\log j_0)^{-\alpha} \sum_{j=j_0}^{[(\log N)^{1/2}]} j^{-\alpha}
+2^{\alpha} (\log \log N)^{-\alpha} \sum_{j=2}^{r_N} j^{-\alpha}, \]
whence we finally get
\begin{equation}\label{esti11}
\Pi_1 \le C  \exp\left( (1+\varepsilon)\left(\frac{8}{1-\alpha}(\log N)^{1-\alpha}(\log\log N)^{-\alpha}+\frac{1}{1-\alpha}(\log N)^{(1-\alpha)/2}\right)  \right),
\end{equation}
assuming again that $j_0$ is sufficiently large.

For the second product in \eqref{th4form}, we obtain
\[ \Pi_3:=\prod_{k=r_N+1}^{N-1}(1-v_{r(N)}^{-1} \tau_k^2)^{-1}\le \exp\left((1+\varepsilon/2)v_{r_N}^{-1}4 \sum_{j=r_N+1}^{N-1}
p_j^{-2\alpha} \right).\]
We appeal again to the prime number theorem and get
\begin{eqnarray}\nonumber \Pi_3 &\le & C\exp\left((1+\varepsilon)4\cdot 2^{\alpha} (2\alpha-1)^{1/2}(\log N)^{\alpha} (\log\log N)^{-\alpha} \sum_{j=r_N+1}^{\infty} j^{-2\alpha}\right) \\
& \le & \label{esti3} C\exp\left((1+\varepsilon)8\cdot 2^{-\alpha} (2\alpha-1)^{-1/2}(\log N)^{1-\alpha} (\log\log N)^{-\alpha} \right)
\end{eqnarray}
The desired estimate for the function $g(\alpha,n)$ in Theorem~\ref{gcd} follows from our three estimates \eqref{esti11}, \eqref{esti2}, and \eqref{esti3}, if we take into account that the contribution from the factors omitted in the first product in \eqref{th4form} by the restriction on $j_0$ can be bounded by a constant $C_\varepsilon$ which is independent of $\alpha$.

The case $\alpha=1/2$ is dealt with in the same way, the only difference being that
we now choose $v_j=\max(\eta( p_j^{-1/2}),  (\log \log N)^{1/2}/(\log N)^{1/2})$. Retaining the notation from the preceding case and assuming that $j_0$ is large enough, we get respectively
\begin{eqnarray*}
\Pi_1 & \le & C  \exp\left( (1+\varepsilon)\left(16(\log N)^{1/2}(\log\log N)^{-1/2}+(\log N)^{1/4}\right)  \right),
 \\
\Pi_2 & \le & C \exp\left((1+\varepsilon)4 (\log N \log\log N)^{1/2}\right), \\
\Pi_3 & \le  & C \exp\left((1+\varepsilon)4  (\log N \log\log N)^{1/2}  \right),
\end{eqnarray*}
where we in the last step used Mertens's second theorem. Combining these estimates, we arrive at the required bound for $g(1/2,N)$ since we may assume that $N$ is so large that $\log\log N\ge 1$.

Finally, to deal with the case $0<\alpha <1/2$, we apply H\"{o}lder's inequality with exponents $1/(2\alpha)$ and  $1/(1-2\alpha)$:
 \[ \frac{1}{N}\sum_{k,{\ell}=1}^N\frac{(\operatorname{gcd}(n_k,n_{\ell}))^{2\alpha}}{(n_k
n_{\ell})^\alpha} \le \left(\sum_{k,\ell=1}^N\frac{\operatorname{gcd}(n_k,n_{\ell})}{(n_k
n_{\ell})^{1/2}} \right)^{2\alpha} N^{1-4\alpha}, \]
and so the desired result follows from what was just proved in the case $\alpha=1/2$.
\end{proof}

To see to what extent Theorem \ref{gcd} is sharp for $1/2\le \alpha < 1$, we consider the following example:  Set $N=2^r$ and take $n_1,...,n_N$ to be all square-free numbers composed of the first $r$ primes. Then
\[ \sum_{k,\ell=1}^N\frac{(\operatorname{gcd}(n_k,n_{\ell}))^{2\alpha}}{(n_k
n_{\ell})^\alpha} =N\prod_{j=1}^r (1+p_j^{-\alpha}),\]
which follows from an argument in \cite[p. 21]{G}. By the prime number theorem, we therefore get
\[ \frac{1}{N}\sum_{k,\ell=1}^{N}\frac{(\operatorname{gcd}(n_k,n_{\ell}))^{2\alpha}}{(n_k
n_{\ell})^\alpha}\ge\exp\left(\frac{c}{1-\alpha}(\log N)^{1-\alpha} (\log\log N)^{-\alpha}\right)\] for some positive constant $c$. Thus our estimate in Theorem~\ref{gcd} is of the right order of magnitude when $1/2<\alpha<1$, as is the blow-up of the multiplicative constant $1/(1-\alpha)$ in t $g(\alpha,N)$ when $\alpha\nearrow 1$. However, this example does not settle the cases $\alpha\searrow 1/2$ and $\alpha=1/2$. In fact, we see that there is a discrepancy of a factor $\log \log N$ in the exponent between our estimate and the lower bound obtained from the example. It seems likely that the blow-up of the constant $c(\alpha)$ when $\alpha\searrow 1/2$ is an artifact. The trouble is that the divergence of the series $\sum_j p_j^{-1}$ implies that the number of primes involved in the sum plays a role. We believe the number of primes should be $\mathcal{O}(\log N)$ when the sum is maximal, but can only infer from our method of proof that this number is bounded by $N-1$.

Our estimate is however essentially optimal when $0<\alpha<1/2$. To see this, it suffices to consider the example $n_1=2, n_2=3,..., n_N=p_{N}$. Using the prime number theorem in a similar way as in the proof of Theorem~\ref{gcd}, we obtain that
\[  \frac{1}{N} \sum_{k,\ell=1}^N\frac{(\operatorname{gcd}(n_k,n_{\ell}))^{2\alpha}}{(n_k
n_{\ell})^\alpha} \ge c (\log N)^{-2\alpha} N^{-2\alpha+1} \]
for a positive constant $c$. The reason for the abrupt change at $\alpha=1/2$ is that the relatively fast divergence of $\sum_j p_j^{-2\alpha}$ (as in this example) plays a dominant role when $0<\alpha<1/2$.

We will now prepare for the proof of Theorem~\ref{general} by making the passage to Poisson integrals as alluded to above. We let $\sigma_K$ denote normalized Lebesgue measure on the unit polycircle $\T^K$ and write
\[ P_K(\zeta,z):=\prod_{k=1}^K \frac{1-|\zeta_k|^2}{|1-\overline{\zeta_k} z_k|^2}, \]
which is the Poisson kernel for the unit polydisc $\D^K$ at the point $\zeta$. It is convenient in this definition to allow $\zeta$ to be a point in the infinite-dimensional polydisc $\D^\infty$. The only property of $P_K$ needed is the identity
\[ t^{|\beta-\mu|}=\int_{\T^K}  z^{\beta}\overline{z}^{\mu} P_K(t,z) d\sigma_K(z), \]
valid for positive sequences $t$ in $\D^\infty$, which is obtained by computing the integral over $\T^K$ as an iterated integral over $K$ copies of the unit circle. It leads immediately to the following lemma.
\begin{lemma}\label{quadratic}
For a positive sequence $t$ in $\D^\infty$, arbitrary multi-indices $\beta_1,...,\beta_N$ with $K=\max_j R(\beta_j)$, and complex numbers $c_1,..., c_N$, we have
\begin{equation}\label{identity} \sum_{k,\ell=1}^N t^{|\beta_k-\beta_{\ell}|} c_k \overline{c}_{\ell} =  \int_{\T^{K}} \Big| \sum_{j=1}^N c_j z^{\beta_j}\Big|^2 P_{K}(t,z) d\sigma_K(z). \end{equation}
\end{lemma}
The fact that the quadratic form on the left-hand side of \eqref{identity} can be written as the square of a norm was first observed in \cite{LS} in the special case when $t=(p_j^{-\alpha})$ and $\alpha>1/2$, based on ideas from \cite{HLS}. The present formulation seems more illuminating and leads to an interesting problem for trigonometric polynomials on $\D^{\infty}$. We will take a closer look at this problem in the next section, where we will estimate the $\ell^2$-norm of the quadratic form on the left-hand side of \eqref{identity}, or, in other words, the largest eigenvalue of the matrix $(t^{|\beta_k-\beta_{\ell}|})$.

For the proof of Theorem~\ref{general}, we only need \eqref{identity} when $c_k\equiv 1$. Incidentally, this restriction is crucial for the combinatorial argument that leads to Lemma~\ref{divisor} below, which is our next auxiliary result. It is interesting to note that this lemma relies on the left-hand side of \eqref{identity}, while the subsequent analytic part of the proof of Theorem~\ref{general} departs from the right-hand side of this identity.

We will use a variant of G\'{a}l's terminology: A set $B$ of $N$ multi-indices $\beta_1,..., \beta_N$ is said to be $\kappa$-canonical for
$0 \le \kappa < N$ if $\beta\in B$ and $e_j \le \beta $ for some $j$ with $\kappa < j \le N$ imply that $\beta-e_j \in B$.  The following lemma is a modification of a theorem in \cite[p. 17]{G}.

\begin{lemma} \label{divisor}
Suppose $B$ is a set of $N$ multi-indices. Let $t$ be a decreasing sequence of positive numbers in $\D^{\infty}\cap c_0$. If $\kappa(t)<N$, then there exists a $\kappa(t)$-canonical set of $N$ multi-indices $B'=\{\beta'_1,...,\beta'_N\}$ such that
$S(\eta( t),B')\ge S(t,B)$ and $\# \bigcup_{j=1}^N \supp \beta_j'\le N-1$.
\end{lemma}

\begin{proof}
We will modify $B$ and $t$ by an inductive algorithm. We break the argument into two parts, the first of which will give a set of multi-indices for which the union of their supports has cardinality at most $N-1$. \vspace{2mm}

\noindent \emph{Part 1}: It will be convenient to use the following terminology. We say that a multi-index $\beta$ in $B$ is $j$-maximal if $j$ is in $\supp \beta$ but $(\beta^{(j)}+1)e_j\not\le \mu$ for every $\mu$ in $B$. We will construct from $B$ a new set $\tilde{B}$ with the property that if $\beta$ in $\tilde{B}$ is $j$-maximal, then also $\beta-e_j$ is in $\tilde{B}$, while at the same time $S(t, \tilde{B})\ge S(t, B)$. Writing $\tilde{B}=\{\tilde{\beta}_1,...,\tilde{\beta}_N\}$, we see that, as a consequence, we will have $\# \bigcup_{j=1}^N \supp \tilde{\beta}_j\le N-1$.

Fix a positive integer $j$ in ${\bigcup}_k \supp \beta_k$.  Let $\nu$ be the largest integer such that $\nu e_j\le \beta$ for some $\beta$ in $B$. Suppose there is a $j$-maximal multi-index $\beta$ in $B$ such that $\nu e_j  \le \beta$ but $\beta-e_j$ is not in $B$. For every such
$\beta$, we replace $\beta$ in $B$ by $\beta-e_j$; we call the new set of multi-indices $B_\nu$.  A term by term comparison shows
that $S(t,B_\nu)\ge S(t,B)$.

If there is a $j$-maximal multi-index in $B_{\nu}$ with $\beta^{(j)}=\mu$, then it must have the desired property that also $\beta-e_j$ is in $B_{\nu}$, and no further action is needed. In the opposite case, we repeat the argument with $\nu$ replaced by $\nu-1$. The iteration terminates when either the desired property holds for some $B_{\eta}$ with $1\le \eta \le \nu$ or $j$ is not in the support of any multi-index in $B_1$.

We repeat this iteration for every $j$ in ${\bigcup}_k \supp \beta_k$ and obtain thus the desired set $\tilde{B}$. \vspace{2mm}

\noindent \emph{Part 2}: By part 1, we may from now on assume that, for every $j$ in $\bigcup_k \supp \beta_k$, any $j$-maximal multi-index $\beta$ in $B$ has the property that $\beta-e_j$ is in $B$. This is irrelevant for the argument to be given below, but we need it to reach the desired conclusion about the cardinality of $\bigcup_j \supp \beta_j$.

We now assume that $\kappa(t)<N$. We fix a $j>\kappa(t)$ in $\bigcup_j \supp \beta_j$ and divide $B$ into disjoint subsets $b_1,..., b_{\ell}$ ($1\le \ell \le N$), which we call $j$-chains of multi-indices, according to the following rule: two distinct multi-indices $\beta$ and $\mu$ belong to the same $j$-chain $b$ if $|\beta-\mu|=\eta e_j$ for some $\eta> 0$. This means that every element $\beta$ in $b$ is of the form $\beta=\mu+\eta e_j$, where $\mu^{(j)}=0$ and $\mu$ is thus a multi-index that characterizes the $j$-chain $b$. We now modify each $j$-chain $b_k$ by replacing it by the set
\[ \tilde{b}_k:=\{ \mu, \mu+e_j,..., \mu+(\# b-1)e_j\}, \]
and we set $\tilde{B}:=\bigcup_{k=1}^{\ell} \tilde{b}_k$.

It is immediate that $S(t, \tilde{b})\ge S(t, b)$. To compare the terms of the sum corresponding to pairs of multi-indices from different $j$-chains, we introduce the notation
\[ S(t; a, b):=\sum_{\beta\in a, \mu\in b} t^{|\beta-\mu|},\]
where $a$ and $b$ are two different $j$-chains. Sorting, by descending order of magnitude, the possible values of $|\beta^{(j)}-\mu^{(j)}|$ for all $\beta$ and $\mu$ in $a$ and $b$ and in $\tilde{a}$ and $\tilde{b}$, respectively, we obtain the inequality
\[ S(t; a,b)\le \sum_{\beta\in \tilde{a}, \mu\in \tilde{b}, \beta^{(j)}=\mu^{(j)}} t^{|\beta-\mu|}
+2\sum_{\beta\in \tilde{a}, \mu\in \tilde{b}, \beta^{(j)}\neq\mu^{(j)}} t^{|\beta-\mu|}.\]
This implies that $S(t; a,b)\le S(t+t_je_j; a,b)$ and, more generally, that
$S(t+t_j e_j, \tilde{B})\ge S(t,B)$.

The result follows if we make this modification in turn for every $j$ in ${\bigcup}_k \supp \beta_k$ for which $j>\kappa(t)$.
\end{proof}

\begin{proof}[Proof of Theorem~\ref{general}]
To simplify the notation, we write $\tau:=\eta(t)$. By Lemma~\ref{divisor}, it suffices to estimate $S(\tau ,B)$ for every $\kappa(t)$-canonical set $B=\{\beta_1,...,\beta_N\}$ of $N$ multi-indices satisfying
\[ \# \bigcup_{j=1}^N \supp \beta_j \le N-1. \]
It is clear that we may assume that
\[  \bigcup_{j=1}^N \supp \beta_j =\{1,2,...,K\} \]
for some $K\le N-1$ since we are seeking an upper bound for all sums $S(\tau,B)$ and $\tau$ is a decreasing sequence.
Note that we may write
\[ P_K(\tau,z)=\prod_{k=1}^K (1-\tau_k^2)\ \Big|\sum_{\beta: R(\beta)\le K} \tau^{\beta}z^{\beta}\Big|^2. \]
By Lemma~\ref{quadratic} and the orthonormality of the monomials $z^{\beta}$, we therefore get
\begin{equation} \label{Piden} S(\tau,B)= \frac{1}{N}\prod_{k=1}^K (1-\tau_k^2)\sum_{\beta: R(\beta)\le K} \left(\sum_{j: \beta_j\le \beta} \tau^{\beta-\beta_j}\right)^2. \end{equation}

Let $\mathcal{B}_1$ denote the set of those multi-indices $\beta$ such that $R(\beta)\le K$ and $\# \supp \beta\le r_N$, and let  $\mathcal{B}_2$ denote the set of all other multi-indices $\beta$ with $R(\beta)\le K$.
By the Cauchy--Schwarz inequality, we get
\[ \sum_{\beta\in \mathcal{B}_2}  \left(\sum_{j:\ \beta_j\le \beta }
\tau^{\beta-\beta_j} \right)^2\le \sum_{\beta \in \mathcal{B}_2}  N \sum_{j:\ \beta_j\le \beta}
\tau^{2(\beta-\beta_j)}, \]
which may be written as
 \[ \sum_{\beta\in \mathcal{B}_2}  \left(\sum_{j:\ \beta_j\le \beta }
\tau^{\beta-\beta_j} \right)^2= \sum_{j=1}^{N}  \sum_{\beta \in \mathcal{B}_2: \beta_j\le \beta} N
\tau^{2(\beta-\beta_j)}.\]
Since $B$ is assumed to be $\kappa(t)$-canonical, $\# \supp \beta_j\le (\log N)/\log 2+\kappa(t)$ for every $j$, and hence
$\# \supp (\beta-\beta_j)\ge \ve \log N$ for a positive $\ve$, depending on our choice of $\xi$, when $\beta$ is in $\mathcal{B}_2$. We assume for convenience that $\ve \log N$ is an integer. Suppose $2\tau_j^2>e^{-1/\ve}$ for $j=1,...,J\le N-1$. Then we may estimate the inner sum as an Euler product and obtain
 \[ \sum_{\beta \in \mathcal{B}_2} N
\tau^{2(\beta-\beta_j)} \le e^{J/\ve}\prod_{j=1}^J(1-\tau_j^2)^{-1}\prod_{k=J}^{N-1} (1-\tau_k^2e^{1/\ve})^{-1},\]
which means that
\begin{equation}\label{b2}  \prod_{k=1}^K (1-\tau_k^2) \sum_{\beta\in \mathcal{B}_2}  \left(\sum_{j:\ \beta_j\le \beta }
\tau^{\beta-\beta_j} \right)^2\le N \exp\big(C\sum_{j=1}^{N-1}t_j^2) \end{equation}
for a constant $C$ that only depends on $\ve$.

We next consider the summation over $\mathcal{B}_1$. Let $\beta$ be an arbitrary multi-index in this set with
\[\supp \beta=\{ j_1,..., j_i\}, \]
where $i\le r_N$ by the definition of  $\mathcal{B}_1$. For any numbers $v_k$ satisfying the hypothesis of Theorem~\ref{general}, we define a sequence $w_{\beta}$ by requiring
\[ w_{\beta}^{(j_k)}:=\begin{cases} v_k & \text{for} \ k=1,...,i \\
0 & \text{otherwise}.
\end{cases} \]
We now apply the Cauchy--Schwarz inequality and get
\begin{eqnarray*}
 \left(\sum_{j:\ \beta_j \le \beta } \tau^{\beta-\beta_j} \right)^2 & \le & \sum_{ j:\ \beta_j \le \beta } w_{\beta}^{\beta-\beta_j}
\sum_{ k:\ \beta_k \le \beta } w_{\beta}^{-(\beta-\beta_k)}\tau^{2(\beta-\beta_k)} \\
& \le & \prod_{j=1}^{r_N}(1-v_j)^{-1}\sum_{ k:\ \beta_k \le \beta } w_{\beta}^{-(\beta-\beta_k)}\tau^{2(\beta-\beta_k)}.
\end{eqnarray*}
Now summing over $\beta$ in $B_1$ and changing the order of summation, we get
\begin{equation}\label{finalstep} \sum_{\beta\in \mathcal{B}_1}\left(\sum_{j:\ \beta_j \le \beta } \tau^{\beta-\beta_j} \right)^2  \le
\prod_{j=1}^{r_N}(1-v_j)^{-1} \sum_{k=1}^N \sum_{\beta\in B_1} w_{\beta}^{-(\beta-\beta_k)}\tau^{2(\beta-\beta_k)}. \end{equation}
Since $(v_j)$ is a nonincreasing sequence, we have
\[ w_{\beta}^{(j)}\ge \begin{cases} v_j & \text{for} \ j\in \supp \beta \cap\{1,...,r_N\} \\
 v_{r_N} & \text{for} \ j\in \supp \beta \cap\{r_N+1,...,N-1\}.
\end{cases} \]
Plugging this estimate into the right-hand side of \eqref{finalstep} and estimating the sum over $\beta\in B_1$ in terms of an Euler product, we conclude that
\[
\sum_{\beta\in \mathcal{B}_1}\left(\sum_{j:\ \beta_j \le \beta } \tau^{\beta-\beta_j} \right)^2  \le N \prod_{j=1}^{r_N}(1-v_j)^{-1}
(1-v_j^{-1} \tau_j^2)^{-1} \prod_{k=r_N+1}^{N-1}(1-v_{r_N}^{-1} \tau_k^2)^{-1}. \]
We finally observe that, in view of \eqref{Piden}, this inequality along with the preceding estimate \eqref{b2} leads to the desired inequality \eqref{th4form}.
\end{proof}

It is worth pointing out that the most essential use of Lemma~\ref{divisor} was to reduce the problem to the case when the cardinalities $\#\supp \beta_j $ are uniformly bounded by a constant times $\log N$. It would be desirable to find a way to arrive at this reduction without involving the auxiliary sequence $\eta(t)$. In particular, if this could be done, then our method of proof would allow us to recapture G\'{a}l's theorem \eqref{galsums}. Unfortunately, we may only conclude from  Theorem~\ref{general} that $\Gamma_{(p_j^{-1})}(N)\ll (\log\log N)^4$.

\section{Spectral norms of generalized GCD matrices}\label{Sect4}

This section will show that we with little extra effort may obtain from Theorem~\ref{general} precise estimates for the largest eigenvalues of the  matrices $(t^{|\beta_k-\beta_{\ell}|})$, which we will refer to as generalized GCD matrices. Since, by \eqref{identity}, these matrices are positive definite, we see that
\[ \Lambda_t(N):=\sup_{\beta_1,...,\beta_N} \sup_{c\neq 0} \frac{\sum_{k,\ell=1}^N t^{|\beta_k-\beta_{\ell}|} c_k \overline{c}_{\ell}}{\sum_{j=1}^N |c_j|^2} \]
is the least upper bound for these eigenvalues, where the suprema are taken over respectively all $N$-tuples of distinct multi-indices $\beta_1,...\beta_N$ and all nonzero vectors $c=(c_1,...,c_N)$ in $\C^N$.  We may also refer to $\Lambda_t(N)$ as the supremum of the
spectral norms of the matrices $(t^{|\beta_k-\beta_l|})$ for fixed $N$. The problem of estimating $\Lambda_t(N)$ for $t=(p_j^{-\alpha})$ was raised in \cite[p. 10]{bewe}. Based on purely arithmetical arguments, Hilberdink \cite[pp. 362--363]{Hi} gave precise estimates for the spectral norms of our GCD matrices in the special case when $p^{\beta_j}=j$ or, in other words, for the matrix corresponding to the first $N$ integers.

Trivially,  $\Lambda_t(N)\ge \Gamma_t(N)$. In the opposite direction, we have the following estimate.

\begin{theorem}\label{eigenvalue}

We have
\[ \Lambda_t(N)\le (e^2+1) ([\log N]+2) \, \max_{1\le n \le N}\Gamma_t (n) \]
whenever $t=(t_j)$ is a decreasing sequence of positive numbers in $\D^{\infty}$. \end{theorem}

A few remarks are in order before we give the proof of this theorem. First, the result is of interest only when $t$ fails to be in $\ell^1$ because if
$t$ is in $\ell^1$, then the easy estimate
\begin{equation} \label{upper} \Lambda_t(N)\le \prod_{j=1}^{N-1} \frac{1+t_j}{1-t_j} \end{equation}
which can be obtained from the right-hand side of \eqref{identity}, will be uniformly bounded when $N\to \infty$. Note that a special version of this estimate is given in \cite[p. 152]{LS}. We will prove both \eqref{upper} and a corresponding estimate for the smallest eigenvalue of $(t^{|\beta_k-\beta_{\ell}|})$ at the end of this section, as a generalization of the result in \cite[p. 152]{LS}.

In our terminology, Dyer and Harman \cite{DH} obtained \eqref{dhsums} from the estimate
\[ \Lambda_{(p_j^{-1/2})}(N)\le C \exp \left( \frac{c \log N}{\log \log N} \right). \]
Besides the results of \cite{LS} and \cite{DH}, we are not aware of previous estimates of $\Lambda_t(N)$ for any other values of $t$.
If we combine Theorem~\ref{gcd} with Theorem~\ref{eigenvalue}, then we obtain precise estimates for $\Lambda_{(p_j^{-\alpha})}(N)$ when
$0<\alpha<1$. From G\'{a}l's theorem \eqref{galsums} and Theorem~\ref{eigenvalue} we also get
\[ \Lambda_{(p_j^{-1})}(N)\le c (\log N)(\log\log N)^2 \]
for an absolute constant $c$. A more subtle application of our estimates for GCD sums, to be given in the next section, will lead to the better bound $ \Lambda_{(p_j^{-1})}(N)\ll (\log \log N)^4$. An interesting point is that this improved estimate is obtained from Theorem~1 and does not require G\'{a}l's theorem.

As an application of our result on spectral norms, we note that we may replace $\lambda_N$ in Theorem 1.1 of \cite[p. 10]{bewe} by our quantity $\Lambda_{(p_j^{-\alpha})}(N)$ and then improve Corollary 1.2 of \cite[p. 11]{bewe} significantly by using our estimates for $\Lambda_{(p_j^{-\alpha})}(N)$.

The phenomenon captured by Theorem~\ref{general} and Theorem~\ref{eigenvalue} is interesting from a function theoretic point of view: While holomorphic polynomials $F$ of fixed $L^2$ norm (in terms of their coefficients) are uniformly bounded at any fixed point in $\D^{\infty}\cap \ell^2$ \cite{CG}, this is not so in general for the Poisson integrals of $|F|^2$. Indeed, the two theorems give a surprisingly precise statement about the relation between the growth of the number of monomials involved in the polynomials  and the growth of such Poisson integrals at points $\zeta$ in the complement of $\D^{\infty}\cap \ell^1$. We believe it could be of interest to clarify how these estimates relate to the distributional properties of polynomial chaos as studied for instance in \cite{Kw}.

Finally, we would like to emphasize the striking point that the combinatorial Lemma~\ref{divisor} seems indispensable in the deduction of our estimates for the spectral norms.

\begin{proof}[Proof of Theorem~\ref{eigenvalue}]
We will estimate the quadratic form
\[ \sum_{k,\ell=1}^N t^{|\beta_k-\beta_{\ell}|} c_k \overline{c}_{\ell} \]
for arbitrary multi-indices $\beta_1,...,\beta_N$ and vectors $c=(c_1,...,c_N)$ satisfying
$ \sum_{j=1}^N |c_j|^2=1$.
We may clearly assume that the coefficients $c_j$ are nonnegative. Set
\[ \mathcal{C}_{\ell}:=\{ j: e^{-\ell-1}<c_j \le e^{-\ell} \}. \]
By the Cauchy--Schwarz inequality, we get
\begin{equation}\label{CS}  \Big| \sum_{j=1}^N c_j z^{\beta_j}\Big|^2 \le ([\log N]+2)\left(\Big| \sum_{j: c_j\le N^{-1}} c_j z^{\beta_j}\Big|^2 + \sum_{\ell: 0 \le \ell < \log N }
 \Big| \sum_{k: k\in \mathcal{C}_{\ell}} c_k z^{\beta_k}\Big|^2\right). \end{equation}
Using \eqref{identity} and again the Cauchy--Schwarz inequality, we get
\[ \int_{\T^{K}} \Big| \sum_{j: c_j\le  N^{-1}} c_j z^{\beta_j}\Big|^2  P_{K}(t,z) d\sigma_K(z) \le 1.\]
Applying \eqref{identity} a second time, we also obtain
\[ \int_{\T^{K}} \Big| \sum_{k: k\in \mathcal{C}_{\ell}} c_k z^{\beta_k}\Big|^2  P_{K}(t,z) d\sigma_K(z) \le e^{-2\ell} (\# \mathcal{C}_{\ell})\, \Gamma_t(\# \mathcal{C}_{\ell}), \]
which, by the definition of $\mathcal{C}_{\ell}$
and the fact that $c$ is a unit vector, implies
\[ \sum_{\ell: 0 \le \ell < \log N } \int_{\T^{K}} \Big|
\sum_{k: k\in \mathcal{C}_{\ell}} c_k z^{\beta_k}\Big|^2  P_{K}(t,z)
d\sigma_K(z) \le e^2\, \max_{1\le n \le N} \Gamma_t(n).\] Returning
to \eqref{CS} and making a final application of \eqref{identity},
we obtain the desired result
\[ \Lambda_t(N)\le ([\log N] +2)(1+e^2)\max_{1\le n \le N}\Gamma_t(n).\]
\end{proof}
Let now $\lambda_t(N)$ denote the infimum of the smallest
eigenvalues of  the generalized GCD matrices
$(t^{|\beta_k-\beta_l|})$ for fixed $N$. We obtain then the
following generalization of the theorem in \cite[p. 152]{LS}.

\begin{theorem}\label{eigenvalue2}
We have
\begin{equation}\label{lowerupper} \prod_{j=1}^{N-1} \frac{1-t_j}{1+t_j}\le
\lambda_t(N)\le \Lambda_t(N)\le \prod_{j=1}^{N-1}
\frac{1+t_j}{1-t_j} \end{equation} whenever $x=(x_j)$ is a
decreasing sequence of positive numbers in  $\D^{\infty}$.
\end{theorem}

\begin{proof} Note first that the expressions to the left and to the right
are respectively the minimum and the maximum of $P_{N-1}(t,z)$ when
$z$ varies  over $\T^{N-1}$. Thus the estimates in
\eqref{lowerupper} follow from \eqref{identity} if we first make the
observation that it suffices to integrate over an $(N-1)$-circle to
compute the $L^2(\sigma_K)$-norm of a function of the form
$\sum_{j=1}^{N} c_j  z^{\beta_j}$. \end{proof}

\section{A Carleson--Hunt-type inequality }
\label{section5}

We have now come to our main application of Theorem~\ref{gcd}, namely to establish a Carleson--Hunt-type inequality.
To this end, we will require the following special case of the classical Carleson--Hunt inequality \cite[Theorem 1]{hunt}.
\begin{lemma} \label{carleson}
There exists an absolute constant $c$ such that
$$
\int_0^1 \left( \max_{1 \leq M \leq N} \left| \sum_{k=1}^M c_k \cos
2 \pi k x  \right| \right)^2 dx \leq c  \sum_{k=1}^N c_k^2
$$
for every finite sequence $(c_k)_{1 \leq k \leq N}$.
\end{lemma}
Our generalized version of this inequality reads as follows (as in the introduction we write $f \in \bv$ for a function which has bounded variation on $[0,1]$).
\begin{lemma} \label{Jlemma}
For every function $f$ satisfying \eqref{f1} and either $f \in \bv$ or $f \in
\lip12$, there exists a constant $c$ such that the following holds. For every finite and strictly increasing sequence of
positive integers $(n_k)_{1 \leq k \leq N}$ and every associated finite sequence of real numbers $(c_k)_{1 \leq k \leq N}$, we have \begin{equation}\label{ch}
\int_0^1 \left( \max_{1 \leq M \leq N} \left| \sum_{k=1}^M c_k f(n_k
x)  \right| \right)^2 dx \leq c \left(\log\log N\right)^4
\sum_{k=1}^N c_k^2.
\end{equation}
\end{lemma}

We do not know whether the exponent of $\log\log N$ is optimal in \eqref{ch}, but the following argument shows that it can not be smaller than 2 for $f$ in $\bv$: If we choose $f(x)=\{x\}-1/2$, then we have the identity \[\int_0^1 f(mx) f(nx) dx= \frac{1}{12} \frac{(\gcd(m, n))^2}{mn},\]
which has been first stated by Franel \cite{franel} and first proved by Landau \cite{landau}. Consequently for this particular function $f$
the left-hand side of \eqref{ch} exceeds
\[ \frac{1}{12} \sum_{k, \ell=1}^N \frac{(\gcd(n_k, n_\ell))^2}{n_k n_\ell} c_k c_\ell.\]
By the optimality of G\'{a}l's theorem \eqref{galsums}, we know that $\Lambda_{(p_j^{-1})}(N)\gg (\log\log N)^2$ in the terminology of the preceding section, and therefore 2 is a lower bound for the exponent. This can also be seen from Hilberdink's computation of the spectral norm of the GCD matrix $((\gcd(m,n))^2/(m,n))_{m,n=1}^N$ (see \cite{Hi}).

The argument just given also shows that Lemma~\ref{Jlemma} implies that $\Lambda_{(p_j^{-1})}(N)\ll (\log\log N)^4$, as announced in the preceding section. Since the maximal operator appearing in Lemma~\ref{Jlemma} is not needed in the computation of the spectral norm, one may suspect that we could do better if our sole goal was to estimate $\Lambda_{(p_j^{-1})}(N)$. However, the proof given below does not give any better bound if we remove the maximal operator on the left-hand side of \eqref{ch}.

Before turning to the proof of Lemma~\ref{Jlemma}, we introduce the following conventions. We write $c$ for appropriate positive
constants, not  always the same, which may depend on $f$, but not on
$N$ or anything else. Any additional dependence is made explicit; we may sometimes, for example, write $c(\ve)$ instead of $c$. We will use the
notation
$$
\left\| g \right\| := \left( \int_0^1 \left( g(x)\right)^2 dx \right)^{1/2},
$$
where $g$ is assumed to be a real-valued function.

\begin{proof}[Proof of Lemma \ref{Jlemma}]
Let $f$ be any function satisfying \eqref{f1}, and assume that either $f \in \bv$ or $f \in \lip12$. To  simplify the exposition, we assume that $f$ is even so that its Fourier series is a pure cosine-series:
$$
f(x) \sim \sum_{j=1}^\infty a_j \cos 2 \pi j x.
$$
Under the assumption that $\sum_k c_k^2=1$, the coefficients $c_k$ satisfying $|c_k|\le N^{-2}$ will give a negligible contribution to the left-hand side of our maximal inequality. We may therefore assume without loss of generality that $N^{-2}\le |c_k|\le 1$.

To make our proof as transparent as possible, we will first prove Lemma \ref{Jlemma} when $f \in \bv$. The proof for $f \in \lip12$ is technically more involved and will be given subsequently. In what follows, we will use the notation
\[ \delta_i=\begin{cases} 1 & \textrm{for $ i=0$} \\
                                       0 & \textrm{otherwise.}\end{cases}
                                       \]

\noindent \emph{Proof in the case $f \in \bv:~$} By \cite[p.~48]{zt}, the Fourier coefficients $a_j$ of a function $f$ in $\bv$  satisfy
\begin{equation} \label{aj2}
|a_j| \leq c j^{-1}, \qquad j \geq 1.
\end{equation}
Set
\begin{equation} \label{pr}
p(x) = \sum_{j=1}^J a_j \cos 2 \pi j x, \qquad r(x) = f(x) - p(x),
\end{equation}
where $J$ will be chosen later. Then, by Minkowski's inequality,
\begin{equation} \label{equ2}
\left\| \max_{1 \leq M \leq N} \left| \sum_{k=1}^M c_k f(n_k x) \right| \right\| \leq \left\| \max_{1 \leq M \leq N} \left| \sum_{k=1}^M c_k p(n_k x) \right| \right\| + \left\| \max_{1 \leq M \leq N} \left| \sum_{k=1}^M c_k r(n_k x) \right| \right\|.
\end{equation}
By \eqref{aj2} and Lemma \ref{carleson}, we have
\begin{eqnarray}
 \left\| \max_{1 \leq M \leq N} \left| \sum_{k=1}^M c_k p(n_k x) \right| \right\| &\le& \sum_{j=1}^J |a_j| \left\| \max_{1 \leq M \leq N} \left| \sum_{k=1}^M c_k \cos 2 \pi j n_k x \right| \right\| \nonumber \\
&  \leq &  c (\log J) \left( \sum_{k=1}^N c_k^2 \right)^{1/2}.   \label{JJ}
 \end{eqnarray}
Estimating the second term on the right-hand side of \eqref{equ2} is more difficult. Let arbitrary numbers $0 \leq M_1 < M_2 \leq N$ be given. We want to find a good estimate for
\begin{equation} \label{asin}
\left\|\sum_{k=M_1+1}^{M_2} c_k r(n_k x) \right\|.
\end{equation}
We now sort the coefficients by size in the same way as we did in the proof of Theorem~\ref{eigenvalue}. Hence, for every $\ell$ in $\left\{0, \left \lceil 2 \log N \right\rceil \right\}$, we define
\begin{equation} \label{kell}
\mathcal{K}_{\ell} := \left\{ k:~M_1 < k \leq M_2 \quad \textrm{and} \quad e^{-\ell-1}<|c_k|\le e^{-\ell} \right\}.
\end{equation}
As observed above, we may assume that $N^{-2} \leq |c_k| \leq 1$ for $1 \leq k \leq N$. Thus
$$
\sum_{\ell=0}^{\left\lceil 2 \log N \right\rceil} \sum_{k \in \mathcal{K}_{\ell}} c_k r(n_k x) = \sum_{k=M_1+1}^{M_2} c_k r(n_kx).
$$
Now let an arbitrary $\ell$ in $\{0, \left\lceil 2 \log N \right\rceil\}$ be fixed, and set $N_{\ell}: = \# \mathcal{K}_{\ell}$. By \eqref{aj2} and the orthogonality of the trigonometric system, we have
\begin{eqnarray}
\int_0^1 \left( \sum_{k \in \mathcal{K}_{\ell}} c_k r(n_k x) \right)^2 dx & = & \frac{1}{2} \sum_{k_1,k_2 \in \mathcal{K}_{\ell}} ~\sum_{j_1,j_2 = J+1}^\infty c_{k_1} c_{k_2} a_{j_1} a_{j_2} ~\delta_{j_1 n_{k_1}-j_2 n_{k_2}} \nonumber \\
& \leq & c e^{-2\ell} \sum_{k_1,k_2 \in \mathcal{K}_{\ell}} ~\sum_{j_1,j_2 = J+1}^\infty (j_1 j_2)^{-1} ~\delta_{j_1 n_{k_1}-j_2 n_{k_2}}. \label{int1}
\end{eqnarray}
Let $v<w$ be two positive integers. Then, following an argument of Koksma \cite{K}, we have
\begin{eqnarray}
\sum_{j_1,j_2 = J+1}^\infty (j_1 j_2)^{-1} ~\delta_{j_1 v-j_2 w} & \leq & \sum_{j_1,j_2 = 1}^\infty (j_1 j_2)^{-1} ~\delta_{j_1 v-j_2 w} \nonumber\\
& = & \sum_{j=1}^\infty \frac{1}{j^2} \frac{\gcd(v,w)}{v} \frac{\gcd(v,w)}{w} \nonumber\\
& \leq & 2~ \frac{\gcd(v,w)^2}{vw}. \label{e1}
\end{eqnarray}
On the other hand, as in \cite[p.~104]{AMZ}, we have
\begin{eqnarray}
\sum_{j_1,j_2 = J+1}^\infty (j_1 j_2)^{-1} ~\delta_{j_1 v-j_2 w} & = & \sum_{j \geq \lceil (J+1) \gcd(v,w)/v \rceil} \frac{(\gcd(v,w))^2}{j^2 vw} \nonumber\\
& \leq & \frac{2}{\lceil (J+1) \gcd(v,w)/v \rceil} \frac{(\gcd(v,w))^2}{vw} \nonumber\\
& \leq & \frac{2}{J} \frac{\gcd(v,w)}{w} \nonumber\\
& \leq & \frac{2}{J} \frac{\gcd(v,w)}{\sqrt{vw}}. \label{e2}
\end{eqnarray}
Let $0<\ve <1$ be a number to be chosen later.
Combining \eqref{e1} and \eqref{e2}, we obtain
\begin{eqnarray}
\sum_{j_1,j_2 = J+1}^\infty (j_1 j_2)^{-1} ~\delta_{j_1 v-j_2 w} & \leq & \left(2~ \frac{\gcd(v,w)^2}{vw}\right)^{1-\ve} \left(\frac{2}{J} \frac{\gcd(v,w)}{\sqrt{vw}}\right)^\ve \nonumber\\
& = & \frac{2}{J^\ve} \frac{\gcd(v,w)^{2-\ve}}{(vw)^{1-\ve/2}}. \label{gcd1}
\end{eqnarray}
Thus the integral in \eqref{int1} is bounded by
$$
c e^{-2\ell} \sum_{k_1,k_2 \in K(\ell)} \frac{2}{J^\ve} \frac{\gcd(n_{k_1},n_{k_2})^{2-\ve}}{(n_{k_1} n_{k_2})^{1-\ve/2}},
$$
which, by Theorem \ref{gcd} (for $\alpha = 1 - \ve/2$), is at most
$$
c e^{-2\ell} J^{-\ve}N_{\ell} \exp\left(\frac{c}{\ve} (\log N_{\ell})^{\ve/2} \right).
$$
By Minkowski's inequality,
we therefore get the following estimate for \eqref{asin}:
\begin{eqnarray*}
\left\| \sum_{k=M_1+1}^{M_2} c_k r(n_k x) \right\|
& \leq & \sum_{\ell=0}^{\lceil 2 \log N \rceil} \left\| \sum_{k \in \mathcal{K}_\ell} c_k r(n_k x) \right\| \\
& \leq & c \sum_{\ell=0}^{\lceil 2 \log N \rceil} e^{-\ell} N_{\ell}^{1/2} J^{-\ve/2} \exp\left(\, \frac{c}{\ve} (\log N_{\ell})^{\ve/2} \right).
\end{eqnarray*}
Applying the Cauchy--Schwarz inequality, we infer from this bound that
\begin{eqnarray}\label{cse}
\left\| \sum_{k=M_1+1}^{M_2} c_k r(n_k x) \right\|
\! \! & \leq & \!\!c J^{-\ve/2} (\log N)^{1/2}\! \left(\sum_{\ell=0}^{\lceil 2 \log N \rceil}\! e^{-2\ell}  N_{\ell} \right)^{1/2}\!\!\! \exp\!\left(\, \frac{\hat{c}}{\ve} (\log N)^{\ve/2} \right) \nonumber\\
& \leq & \!c J^{-\ve/2} (\log N)^{1/2} \left(\sum_{k=M_1+1}^{M_2} c_k^2 \right)^{1/2} \exp\left( \frac{\hat{c}}{\ve} (\log N)^{\ve/2} \right). \label{asin2} 
\end{eqnarray}
The constant $\hat{c}$ in (\ref{asin2}) is marked by $\hat{~}$ to indicate that its value (unlike the value of the other constants denoted by $c$) does not change in the sequel. Without loss of generality, we may assume that $\hat{c}\ge 4$.  We now choose $J$ by requiring that
\begin{equation}\label{Jdef}
J^{\ve/2}= (\log N)^{1/2}\exp \left(\frac {2\hat{c}}{\ve}(\log N)^{\ve/2}\right)
\end{equation}
so that (\ref{asin2}) becomes
$$\left\| \sum_{k=M_1+1}^{M_2} c_k r(n_k x) \right\| \le c\left( \sum_{k=M_1+1}^{M_2} c_k^2\right)^{1/2} \exp \left(-\frac {\hat{c}}{\ve}(\log N)^{\ve/2}\right).$$
Now imitating the proof of the Rademacher--Menshov inequality
(see \cite[p.~123]{loeve}), we see  that this estimate implies
\begin{eqnarray}\label{rame}
\left\| \max_{1 \leq M \leq N} \left| \sum_{k=1}^{M} c_k r(n_k x) \right| \right\| & \leq & c \log N \exp\left(- \frac{\hat{c}}{\ve} (\log N)^{\ve/2} \right) \left(\sum_{k=1}^{N} c_k^2 \right)^{1/2}. 
\end{eqnarray}
Choosing $\ve=1/(\log\log N)$ and recalling that $\hat{c}\ge 4$, we see that
the expression in (\ref{rame}) will be bounded by $c(\sum_{k=1}^N c_k^2)^{1/2}$. On the other hand,
\begin{equation}\label{ca}
\log J=\frac{1}{\ve} \log\log N+ \frac{4\hat{c}}{\ve^2} (\log N)^{\ve/2},
\end{equation}
which is less than or equal to $ c(\log\log N)^2$ with our choice of $\ve$. Thus \eqref{JJ}  becomes
$$ \left\| \max_{1 \leq M \leq N} \left| \sum_{k=1}^M c_k p(n_k x) \right| \right\|\le c (\log\log N)^2 \left( \sum_{k=1}^{N} c_k^2\right)^{1/2},
$$
which, together with \eqref{rame}, proves the lemma in the case $f \in \bv$.\vspace{2mm}

\noindent \emph{Proof in the case $f \in \lip12$~}: If $f \in \lip12$, then by \cite[p.~241]{zt} we have
\begin{equation} \label{fcoeff}
\sum_{j=2^m+1}^{2^{m+1}} a_j^2 \leq c 2^{-m}, \qquad  m \geq 0.
\end{equation}
Note that if $f \in \bv$, then \eqref{fcoeff} also holds as a consequence of \eqref{aj2}; thus the proof for the case $f \in \bv$ could have been included in the present proof. However, \eqref{fcoeff} is significantly weaker than \eqref{aj2}, which makes the proof in the present case more complicated. By the Cauchy--Schwarz inequality, \eqref{fcoeff} implies that
$$
\sum_{j=2^m+1}^{2^{m+1}} \left| a_j \right| \leq c,
$$
and hence
\begin{equation} \label{J2}
\sum_{j=1}^J |a_j| \leq c \log J
\end{equation}
for any $J \geq 1$.
Define $p,r$ as in (\ref{pr}), with $J$ to be chosen later.   
We estimate the second term on the right-hand side of \eqref{equ2}. To this end, assume that $0<\ve <1$, and set
$$
\mathcal{S}_{m}: = \left\{ 2^m < j \leq 2^{m+1}:~|a_j| \leq 2^{-m(1-\ve)} \right\}, \qquad \mathcal{T}_m := \{2^m+1, \dots, 2^{m+1}\} \backslash \mathcal{S}_{m}.
$$
Then from \eqref{fcoeff} it is clear that
\begin{equation} \label{tgr}
\# \mathcal{T}_m \leq c 2^{m-2m\ve}.
\end{equation}
Let $0 \leq M_1 < M_2 \leq N$ be given, and let $\mu$ denote the largest integer such that $2^\mu \leq J$. Replacing all coefficients by their absolute values (which is permitted due to the orthogonality of the trigonometric system), starting the summation at $2^\mu$ instead of $J$ and applying Minkowski's inequality twice we get
\[
 \left\| \sum_{k=M_1+1}^{M_2} c_k r(n_k x) \right\|
 \leq  \sum_{m=\mu}^\infty \left\| \sum_{k=M_1+1}^{M_2} \sum_{j=2^m+1}^{2^{m+1}} |a_j|~ |c_k| \cos 2 \pi j n_k x\right\|
 \qquad\qquad\qquad\qquad\qquad\]
\[ \qquad \le \sum_{m=\mu}^\infty \left( \biggl\| \sum_{k=M_1+1}^{M_2} \sum_{j \in \mathcal{S}_m} |a_j|~ |c_k| \cos 2 \pi j n_k x\biggr\| +  \biggl\| \sum_{k=M_1+1}^{M_2} \sum_{j \in \mathcal{T}_m} |a_j| ~|c_k| \cos 2 \pi j n_k x\biggr\| \right). \]
We reverse the order of summation   and use Minkowski's inequality along with  \eqref{tgr}, \eqref{fcoeff}, and the orthogonality of the trigonometric system to estimate the second norm on the right-hand side of this inequality. Using also the definition of $\mathcal{S}_m$ to deal with the first norm, we therefore get:
\begin{equation}
\left\| \sum_{k=M_1+1}^{M_2} \! \!c_k r(n_k x) \right\|
\!\le\!
 \sum_{m=\mu}^\infty \!\!\left( \biggl\| \sum_{k=M_1+1}^{M_2} \sum_{j \in \mathcal{S}_m}\! j^{-1+\ve} |c_k| \cos 2 \pi j n_k x\biggr\|\! \!+\! c2^{-m \ve}\!\!\left( \sum_{k=M_1+1}^{M_2} \!\! c_k^2 \right)^{1/2}\! \right)\!. \label{equr}
\end{equation}
Now let $m$ be fixed. We define $\mathcal{K}_{\ell}$ as in \eqref{kell}, and observe that
\begin{equation} \label{left}
\int_0^1\! \left( \sum_{k \in \mathcal{K}_{\ell}} \sum_{j \in \mathcal{S}_m} j^{-1+\ve} |c_k| \cos 2 \pi j n_k x  \right)^2\! dx
 \leq  c e^{-2\ell} \!\sum_{k_1,k_2 \in \mathcal{K}_{\ell}} \sum_{j_1,j_2 = 2^m+1}^\infty (j_1 j_2)^{-1+\ve} \delta_{j_1 n_{k_1}- j_2 n_{k_2}}. \end{equation}
Instead of \eqref{e1}, we get
\begin{eqnarray}
\sum_{j_1,j_2 = 2^m+1}^\infty (j_1 j_2)^{-1+\ve} ~\delta_{j_1 v - j_2 w} & \leq & \sum_{j=1}^\infty \frac{1}{j^{2-2\ve}} \left(\frac{\gcd(v,w)}{v} \frac{\gcd(v,w)}{w} \right)^{1-\ve} \nonumber\\
& \leq & c~ \frac{\gcd(v,w)^{2-2\ve}}{(vw)^{1-\ve}},\label{e1a}
\end{eqnarray}
and as a replacement for \eqref{e2}, we have
\begin{eqnarray}
\sum_{j_1,j_2 = 2^m+1}^\infty (j_1 j_2)^{-1+\ve} ~\delta_{j_1 v - j_2 w} & = & \sum_{j \geq \lceil (2^m+1) \gcd(v,w)/v \rceil} \frac{(\gcd(v,w))^2}{j^{2-2\ve} vw} \nonumber\\
& \leq & \frac{c}{2^{m(1-2\ve)}} \frac{(\gcd(v,w))^{1+2\ve}}{(vw)^{1/2+\ve}}. \label{e2a}
\end{eqnarray}
Combining \eqref{e1a} and \eqref{e2a} with exponents $1-2\ve$ and $2\ve$, respectively, we have
\begin{eqnarray*}
\sum_{j_1,j_2 = 2^m+1}^\infty (j_1 j_2)^{-1+\ve} ~\delta_{j_1 v - j_2 w} & \leq & c \left(\frac{\gcd(v,w)^{2-2\ve}}{(vw)^{1-\ve}} \right)^{1-2\ve} \left(\frac{1}{2^{m(1-2\ve)}} \frac{(\gcd(v,w))^{1+2\ve}}{(vw)^{1/2+\ve}}\right)^{2\ve} \\
& \leq & c 2^{-2 m \ve(1-2\ve)}  \left(\frac{\gcd(v,w)^2}{vw}\right)^{1-2\ve+4\ve^2} \\
& \leq & c 2^{-m\ve}  \left(\frac{\gcd(v,w)^2}{vw}\right)^{1-\ve} \\
\end{eqnarray*}
(where we assume w.l.o.g. that $\ve \leq 1/5$), and consequently \eqref{left} becomes
$$
\int_0^1\! \left( \sum_{k \in \mathcal{K}_{\ell}} \sum_{j \in \mathcal{S}_m} j^{-1+\ve} |c_k| \cos 2 \pi j n_k x  \right)^2\! dx \le
c e^{-2\ell} \sum_{k_1,k_2 \in K(\ell)} 2^{-m \ve} \frac{(\gcd(n_{k_1},n_{k_2}))^{2-\ve}}{(n_{k_1} n_{k_2})^{1-\ve/2}}.
$$
As in \eqref{asin2}, we therefore obtain the upper bound
\[
 \left\| \sum_{k=M_1+1}^{M_2} \sum_{j \in \mathcal{S}_m} j^{-1+\ve} |c_k| \cos 2 \pi j n_k x\right\| \qquad \qquad\qquad\qquad\qquad\qquad
 \qquad\qquad\qquad\qquad \qquad\]
 \begin{equation}
 \qquad\leq \ \ c 2^{-m\ve/2} (\log N)^{1/2} \left(\sum_{k=M_1+1}^{M_2} c_k^2 \right)^{1/2} \exp\left( \frac{c}{\ve} (\log N)^{\ve/2} \right). \label{upperb}
\end{equation}
Along with \eqref{equr} this yields
\[
 \left\| \sum_{k=M_1+1}^{M_2} c_k r(n_k x) \right\|  \leq \ c J^{-\ve/2} \left(\log N\right)^{1/2} \left(\sum_{k=M_1+1}^{M_2} c_k^2 \right)^{1/2} \exp\left(\frac{c}{\ve} (\log N)^{\ve/2} \right),
\]
which is identical to \eqref{cse}. Hence the rest of the proof can be carried out as in the case when $f \in \bv$.
\end{proof}

\begin{proof}[Proof of Theorem \ref{a1} and Theorem \ref{a2}]
Assuming the validity of Theorem \ref{a2}, the series (\ref{fseries}) converges a.e.\ for
any $(n_k)_{k\ge 1}$ and $c_k= (k\log k)^{-1/2} (\log\log k)^{-(5/2+\varepsilon)}$ ($\varepsilon>0$)
and thus by the Kronecker lemma, (\ref{1/2est}) is valid. Thus Theorem \ref{a1} follows from
Theorem \ref{a2}, and it suffices to prove Theorem \ref{a2}. Let $(n_k)_{k\ge 1}$ be
an increasing sequence of integers and $(c_k)_{k \geq 1}$ a sequence
of real numbers such that for some $\delta>0$ we have
\begin{equation*}
\sum_{k=1}^\infty c_k^2 (\log \log k)^{4 +\delta} < \infty.
\end{equation*}
Let $N_m$ be an increasing sequence of integers such that
$$ \log\log N_m \sim m^\gamma \qquad \text{with} \ \gamma\ge 6/\delta.$$
Clearly
$$
\sum_{k=N_m+1}^{N_{m+1}}c_k^2 \le (\log\log N_m)^{-(4+\delta)}\sum_{k=N_m+1}^{N_{m+1}}c_k^2
(\log\log k)^{4+\delta}\le c(\log\log N_m)^{-(4+\delta)}
$$
and thus by Lemma \ref{Jlemma} and the Chebyshev inequality we get, writing $\lambda$ for the Lebesgue measure,
\begin{eqnarray*}
& & \lambda \left( \left\{ x \in (0,1):~ \max_{N_m+1\le M\le N_{m+1}} \left| \sum_{k=N_m+1}^M c_k
f(n_k x)\right| \ge m^{-2} \right\} \right) \\
&  \le & c m^4
\left(\sum_{k=N_m+1}^{N_{m+1}} c_k^2\right)
(\log\log N_{m+1})^4 \\
&\le & c m^4 \left(\sum_{k=N_m+1}^{N_{m+1}} c_k^2\right)(\log\log
N_m)^4  \\ &\le& c m^4(\log\log N_m)^{-\delta}\le cm^{-2}.
\end{eqnarray*}
We set $S_N(x):=\sum_{k=1}^N c_k f(n_kx)$ and see that the latter estimate, along with the Borel--Cantelli lemma, yields
\begin{equation}\label{Snk}
\max_{N_m \le M \le N_{m+1}} |S_M-S_{N_m}| =\max_{N_m+1\le M\le N_{m+1}} \left| \sum_{k=N_m+1}^M c_k
f(n_k x)\right| \ll m^{-2} \qquad \text{a.e.}
\end{equation}
In particular, $\sum_{m=1}^\infty |S_{N_{m+1}}-S_{N_m}|<\infty$ a.e., which implies the a.e.\ convergence of $S_{N_m}$. Using \eqref{Snk}, we finally obtain the a.e.\ convergence of $S_N$. \end{proof}

\section{Divergence of series involving dilations of $\{x\}-1/2$}\label{section6}
We finally turn to the example showing that Theorem~\ref{a2} is essentially best possible for the class $\bv$. In what follows, we will use the notation $\varphi(x):=\{x\}-1/2.$ Our arguments will be probabilistic and
we will use the symbols $\mathbb{P}$ and $\mathbb{E}$ with respect to the unit interval equipped with Borel sets
and the Lebesgue measure.

\begin{theorem}\label{ex}
For every $0<\gamma< 2$, there exists
an increasing sequence $(n_k)_{k \geq 1}$ of positive integers and a real
sequence $(c_k)_{k \geq 1}$  such that
\[ \sum_{k=1}^{\infty} c_k^2 (\log \log k)^{\gamma} <\infty, \]
but $\sum_{k=1}^\infty
c_k \varphi(n_kx)$ is a.e.\ divergent.
\end{theorem}

We will need the following variant of Lemma 2 of \cite{berkes}.

\begin{lemma}\label{ber} Let $1\le p_1<q_1<p_2<q_2<\dots$ be integers such that
$p_{m+1}\ge 16q_m$; let $I_1, I_2,\dots$ be sets of integers such
that $I_m\subset [2^{p_m}, 2^{q_m}]$ and each element of $I_m$ is
divisible by $2^{p_m}$. For $m \geq 1$ and $\omega \in (0,1)$ set
$$
X_m= X_m(\omega):=\sum_{k \in I_m} \varphi(k \omega).
$$
Then there exist independent random variables $Y_1,Y_2,\dots$ on the
probability space $((0,1),\mathcal B, {\mathbb P})$ such that $|Y_k|\le
\text{card}\, I_k$, $\mathbb{E}Y_k=0$ and
$$
\|X_m-Y_m\|\le 2^{-m}\qquad \textrm{for $m\ge m_0$},
$$
where $\| \cdot \|$ denotes the $L^2(0, 1)$ norm.
\end{lemma}

\begin{proof}
 Let ${\mathcal F}_m$ denote the $\sigma$-field generated by
the dyadic intervals
\begin{equation}\label{4}
U_j:=\left[j 2^{-16q_m},  (j +1)2^{-16q_m}\right],\qquad 0\le j
< 2^{16q_m}
\end{equation}
and set
\begin{align*}
& \xi_k = \xi_k(\cdot) :=\mathbb{E}(\varphi(k \cdot)|{\mathcal F}_m), \qquad k\in I_m \\
& Y_m=Y_m(\omega):  =\sum_{k\in I_m} \xi_k(\omega).
\end{align*}
Since $|\varphi|\le 1$, we have $|\xi_k|\le 1$ and thus $|Y_m|\le
\text{card}\, I_m$. Further, by $\varphi \in \text{BV}$ the Fourier
coefficients of $\varphi$ are $\mathcal{O}(1/k)$ and thus from Lemma 3.1 of
\cite{be76} it follows that
$$
\|\xi_k (\cdot) -\varphi(k\cdot)\| \ll (k 2^{-16q_m})^{1/6}\qquad k \in
I_m,
$$
and since $I_m$ has at most $2^{q_m}$ elements, we get
$$
\|X_m-Y_m\|\ll 2^{-{q_m}},
$$
which implies
$$
\|X_m-Y_m\| \le 2^{-m}
\qquad\text{for}\qquad m \ge m_0.
$$
Since $p_{m+1}\ge 16q_m$ and since each $k\in I_{m+1}$ is a multiple
of $2^{p_{m+1}}$, each interval $U_j$ in (\ref{4}) is a period
interval for all $\varphi(k x)$, $k \in I_{m+1}$ and thus also for $\xi_k$,
$k \in I_{m+1}$. Hence $Y_{m+1}$ is independent of the $\sigma$-field
${\mathcal F}_m$, and since ${\mathcal F}_1\subset {\mathcal F}_2
\subset \dots$ and $Y_m$ is ${\mathcal F}_m$-measurable, the random variables
$Y_1, Y_2, \dots$ are independent. Finally $\mathbb{E}\xi_k=0$ and thus
$\mathbb{E} Y_m=0$.
\end{proof}

\begin{proof}[Proof of Theorem~\ref{ex}]

We will actually prove a little more than what is stated in the theorem: we show that for any positive
sequence $\ve_k \to 0$ there exists an increasing sequence $(n_k)_{k \geq 1}$
of integers and a real sequence $(c_k)_{k \geq 1}$ such that
$$ \sum_{k=1}^\infty c_k^2 (\log\log k)^2 \ve_k<\infty$$
and $\sum_{k=1} ^\infty c_k \varphi(n_k x)$ diverges a.e. Let $\ve_k^*=\sup_{j\ge k} \ve_j$ and let $(\psi_k)_{k \geq 1}$ be a sequence of positive integers growing so rapidly that
$\psi_{k+1}/\psi_k\ge 2$ for $k \geq 1$ and
$$\sum_{m=1}^\infty \ve_{M_{m-1}}^*<\infty$$
where
$$M_m := \sum_{k \leq m} \psi_k^4. $$
Put $r_k:=\psi_k^3$. By the result of G\'al \cite{G} stated in the introduction, there exists, for each
$m \ge 1$, a sequence $n_1^{(m)} < n_2^{(m)} <\ldots <
n_{\psi_m}^{(m)}$ of positive integers such that
\begin{equation}\label{7}
\int\limits_0^1 \left( \sum_{k=1}^{\psi_m}  \varphi (n_k^{(m)} \omega)
\right)^2 d\omega \ge c \psi_m (\log\log \psi_m)^2
\end{equation}
(here, and in the sequel, $c$ denotes appropriate positive constants, not always the same).
Note that by the upper estimate in G\'al's theorem \cite{G}, the opposite inequality in
(\ref{7}) with a suitable $c$ is automatically valid.
We define sets
\begin{equation}\label{8}
I_1^{(1)},I_2^{(1)},\dots,I_{r_1}^{(1)},
I_1^{(2)},\dots,I_{r_2}^{(2)}, \dots, I_1^{(m)},\dots,I_{r_m}^{(m)},
\dots
\end{equation}
of positive integers by requiring
$$
I_k^{(m)}:= 2^{a_k^{(m)}} \left \{n_1^{(m)},\dots, n_{\psi_m}^{(m)}
\right\},\qquad 1\le k \le r_m,\; m\ge 1,
$$
where $a_k^{(m)}$ are suitable positive integers. (Here for any set
$\{ a,b,\dots\}\subset \mathbb{R}$ and $\mu \in \mathbb{R}$ we write $\mu \{
a,b,\dots\}$ for the set $\{ \lambda a,\lambda b,\dots\}$.)
Clearly we can choose the integers $a_k^{(m)}$ inductively so that
the sets $I_k^{(m)}$ in (\ref{8}) satisfy the conditions assumed in Lemma \ref{ber}
for the sets $I_m$. Since the left-hand side of (\ref{7})
does not change if we replace every $n_k^{(m)}$ with $a n_k^{(m)}$
for some integer $a \geq 1$, setting
$$
X_k^{(m)} = X_k^{(m)} (\omega):= \sum_{j \in I_k^{(m)}} \varphi
(j \omega),
$$
then we have
\begin{equation}\label{9}
\mathbb{E} \left( X_k^{(m)}\right)^2 \ge c \psi_m (\log\log
\psi_m)^2.
\end{equation}
Note that, just as in
the case of (\ref{7}), the opposite inequality with a suitable $c$ is also valid in (\ref{9}).
By Lemma \ref{ber}, there exist independent random variables $Y_k^{(m)}$
($1\le k \le r_m$, $m \geq 1$), such that $|Y_k^{(m)}|\le \psi_m$,
$\mathbb{E} Y_k^{(m)}=0$ and
\begin{equation}\label{10}
\sum_{m,k} \|X_k^{(m)}-Y_k^{(m)}\|<\infty \qquad
\end{equation}
whence
\begin{equation}\label{11}
\sum_{m,k} |X_k^{(m)}-Y_k^{(m)}|<\infty \quad \text{a.e.}
\end{equation}
By (\ref{9}) and (\ref{10}) we have
\begin{equation}\nonumber 
\mathbb{E} \left(Y_k^{(m)}\right)^2 \ge c \psi_m (\log\log \psi_m)^2.
\end{equation}
Hence setting
$$
Z_m: = \dfrac{1}{\sqrt{r_m\psi_m} \log\log \psi_m}
\sum_{k=1}^{r_m} Y_k^{(m)}, \qquad \sigma_m^2: = \mathbb{E}\left(
\sum_{k=1}^{r_m} Y_k^{(m)}\right)^2,
$$
we get from the central limit theorem with Berry--Esseen remainder
term (see e.g. \cite[p. 544]{fe}), (7), and $r_m=\psi_m^3$, that
\begin{eqnarray*}
\mathbb{P} (Z_m \ge 1) & \ge & \mathbb{P} \left( \sum_{k=1}^{r_m} Y_k^{(m)} \ge  c_1
\sigma_m \right)
\ \ge \
 1-\Phi (c_1) - c \dfrac{r_m \psi_m^3}{(r_m \psi_m
(\log\log \psi_m)^2)^{3/2}} \\
& \ge & 1-\Phi (c_1)- o(1)\ \ge \ c_2 >0\qquad  \textrm{for $m \ge m_0$},
\end{eqnarray*}
where $\Phi$ denotes the Gaussian distribution function and $c_1$ and
$c_2$ are positive absolute constants. Since the random variables $Z_m$ are
independent, the Borel--Cantelli lemma implies that $\mathbb{P}(Z_m\ge 1 \;
\text{for infinitely many}\ m )=1$ and consequently $\sum_{m=1}^\infty Z_m$ is a.e.\ divergent, which,
in view of (\ref{11}), yields that
\begin{equation}\nonumber 
\sum_{m=1}^\infty\, \dfrac{1}{\sqrt{r_m \psi_m} \log\log \psi_m}\,
\sum_{k=1}^{r_m} X_k^{(m)} \qquad\text{is a.e.\ divergent.}
\end{equation}
In other words, $\sum_{k=1}^\infty c_k \varphi(n_k x)$ is a.e.\ divergent,
where
\begin{equation}\nonumber 
(n_k)_{k\ge 1}:= \bigcup\limits_{m=1}^\infty\,
\bigcup\limits_{k=1}^{r_m} I_k^{(m)}
\end{equation}
and
$$ c_k^2: =  \frac{1}{r_m \psi_m
(\log\log \psi_m)^2} \qquad \text{for} \quad M_{m-1} < k \leq
M_m.$$ Now for $M_{m-1} < k \leq M_m$ we have by the exponential growth of $(\psi_k)_{k\ge 1}$ with quotient
$q\ge 2$  that
$$k \le 2\psi_m^4 \qquad \textrm{and} \qquad \log\log k \le
2\log\log \psi_m \qquad \textrm{for $m\ge m_0$}.$$ Consequently for $M_{m-1}
< k \leq M_m$ we have
$$ c_k^2 (\log\log k)^2 \ve_k \leq c \frac{1}{r_m \psi_m} \ve_{M_{m-1}}^*. $$ Hence
$$ \sum_{k=1}^{\infty} c_k^2 (\log\log k)^2 \varepsilon_k \leq c \sum_{k=1}^{\infty}
\ve_{M_{k-1}}^*  < \infty,$$ which means that we have reached the desired conclusion.

\end{proof}

\section*{Acknowledgements} The authors are grateful to Eero Saksman for a careful reading of the manuscript and in
particular for pertinent remarks concerning Lemma~\ref{divisor}. They would also like to express their gratitude to the anonymous referee for a careful review leading to a clarification of some essential technical details.

\end{document}